\documentclass[a4paper,10pt]{amsart}

\usepackage[textwidth=125mm, textheight=195mm]{geometry}
\usepackage{amssymb}
\usepackage{amsthm}  
\usepackage{amsmath} 
\usepackage{amscd} 
\usepackage[all]{xypic}
\usepackage{url}
\usepackage{color}

\newcommand{\om}{\omega}

\newcommand{\ba}{\mathcal{G}}

\newcommand{\fg}{\mathfrak g}

\newcommand{\fp}{\mathfrak p}
\newcommand{\fq}{\mathfrak q}

\newcommand{\fl}{\mathfrak l}
\newcommand{\fb}{\mathfrak b}
\newcommand{\fn}{\mathfrak n}

\newcommand{\Ad}{{\rm Ad}}

\newcommand{\id}{{\rm id}}

\newcommand{\R}{\mathbb{R}}

\newtheorem{prop*}{Proposition}

\newtheorem{thm*}{Theorem}

\newtheorem{lemma*}{Lemma}

\newtheorem{cor*}{Corollary}

\newtheorem{rem*}{Remark}
\theoremstyle{definition}
\newtheorem{def*}{Definition}

\theoremstyle{remark}

\begin{document}
\title{A construction of non--flat non--homogeneous symmetric parabolic geometries}
\author{Jan Gregorovi\v c and Lenka Zalabov\' a}
\address{J.G. E.\v Cech  Institute,  Mathematical  Institute  of Charles  University,  Sokolovsk\'a 83, Praha 8 - Karl\'in, Czech Republic; L.Z. Institute of Mathematics and Biomathematics, Faculty of Science, University of South Bohemia, Brani\v sovsk\' a 1760, \v Cesk\' e Bud\v ejovice, 370 05, Czech Republic}
\email{jan.gregorovic@seznam.cz, lzalabova@gmail.com}
\thanks{First author supported by the grant P201/12/G028 of the Czech Science Foundation.}
\dedicatory{
The paper is dedicated to Professor Vladim\' ir Sou\v cek on the occasion of his 70th birthday.}
\maketitle

\begin{abstract} 
We construct a series of examples of non--flat non--homogeneous parabolic geometries that carry a symmetry of the parabolic geometry at each point.
\end{abstract}

\section{Introduction} \label{prvni-cast}
In this article, we deal with symmetric regular and normal parabolic geometries on smooth connected manifolds. Consider a regular and normal parabolic geometry $(\ba \to M, \om)$ of type $(G,P)$. A \emph{symmetry} at a point $x\in M$ is an automorphisms $\phi_x$ of the parabolic geometry such that $\phi_x(x)=x$ and the restriction of $T_x\phi_x$ to the bracket generating distribution $T^{-1}M$ is $-\id$. The parabolic geometry is \emph{symmetric}, if there is a symmetry at each $x \in M$. 

There are several known constructions of examples of symmetric parabolic geometries. In particular, there is a simple condition proved in \cite{GZ-Lie} that is necessary and sufficient for the existence of symmetric parabolic geometries.

\begin{lemma*}
Let $G$ be semisimple Lie group and $P$ parabolic subgroup of $G$. Let $G_0\ltimes \exp(\fp_+)$ be the reductive Levi decomposition of $P$ corresponding to grading $\fg_i$ of $\fg$, where $\fg_0$ is the Lie algebra of $G_0$ and $\fp_+=\fg_1\oplus \dots \oplus\fg_k$.

If the parabolic geometry $(\ba \to M, \om)$ of type $(G,P)$ is symmetric, then there is $s\in G_0$ acting as $-\id$ on $\fg_{-1}$. Moreover, if the type $(G,P)$ is effective, then the element $s$ is unique element of  $G_0$ acting as $-\id$ on $\fg_{-1}$.

Conversely, if there is $s\in G_0$ acting as $-\id$ on $\fg_{-1}$, then the flat model $(G \to G/P, \om_G)$ is symmetric. In particular, there is infinite number of symmetries at the origin $eP$ given by left multiplications by elements of the form 
$$s\exp(-\Ad_s(Y))\exp (Y)$$ 
for $Y \in \fp_+$ and symmetries at arbitrary point $gP$ are then given by conjugation $gs\exp(-\Ad_s(Y))\exp (Y) g^{-1}$. In fact, we get a symmetric flat homogeneous parabolic geometry.
\end{lemma*}
It is proved in \cite[Proposition 1.29]{HG-DGA} that for each semisimple Lie algebra $\fg$ and its parabolic subalgebra $\fp$, there always exist a Lie group $G$ and its closed subgroup $P$ such that the flat model $(G \to G/P, \om_G)$ is symmetric. In fact, there is a general construction of flat and non--flat homogeneous symmetric parabolic geometries on homogeneous fiber bundles over symmetric spaces described in the article \cite[Theorem 2.7.]{HG-DGA}.

There are also examples of flat non--homogeneous symmetric parabolic geometries obtained from the flat model, which are not related to symmetric spaces. It is shown in \cite{ja-springer, ja-CEJM, GZ-Srni15} that if we remove two distinguished points $u,v$ from the flat model $(G \to G/P, \om_G)$ of parabolic geometries of projective, projective contact and conformal type, then the restrictions of the flat models $(G \to G/P, \om_G)$ to $M:=G/P -\{u,v\}$ are still symmetric parabolic geometries. In all these cases, the manifold $M$ decomposes into several orbits with respect to the action of the automorphisms group (which consists exactly of elements of $G$ that preserve the subset $\{ u,v\} \subset G/P$), and on each of this orbits, the symmetries either preserve $u$ and $v$ or swap them.

Further, there are constructions of homogeneous symmetric parabolic geometries other than the construction in \cite{HG-DGA}. In particular, in \cite{GZ-Lie} is presented a construction of non--flat homogeneous symmetric parabolic geometries on a (semidirect) product of a flat model of a different (non--effective) type of parabolic geometry and a homogeneous space of a nilpotent Lie group. 

There is a natural question, whether there are also non--flat non--homogeneous symmetric parabolic geometries? It is proved in \cite{GZ-Srni15} that all non--homogeneous symmetric conformal geometries are necessarily flat and it is clear from the proof that the same result can be obtained for all AHS--structures. However, we will show in this article that we can combine the constructions from  \cite{GZ-Lie} and \cite{ja-springer, ja-CEJM, GZ-Srni15} and prove that there are types of parabolic geometries for which the question can be answered positively. 

In the first section, we show how to combine the above constructions to get new examples of non--flat non--homogeneous symmetric parabolic geometries. We discuss several necessary and sufficient conditions under which the construction is applicable. As our main result, we show in the Theorem \ref{main} that there are two series of non--flat non--homogeneous symmetric parabolic geometries provided by our construction. We describe these parabolic geometries in detail.

In the second section, we give a proof of the main Theorem \ref{main}. The proof consists of several technical lemmas and we explain the technicalities in detail. 

\section{Non--flat non--homogeneous symmetric parabolic geometries}
Let us firstly give the statement that explains how to combine the two constructions of symmetric parabolic geometries mentioned in the Introduction.
\begin{prop*} \label{prop}
Let $G$ be semisimple Lie group and $P$ parabolic subgroup of $G$. Let $G_0\ltimes \exp(\fp_+)$ be the reductive Levi decomposition of $P$ corresponding to grading $\fg_i$ of $\fg$, where $\fg_0$ is the Lie algebra of $G_0$ and $\fp_+=\fg_1\oplus \dots \oplus\fg_k$. Suppose there is a non--flat $K$--homogeneous parabolic geometry $(\ba\to M,\om)$ of type $(G,P)$ satisfying the following conditions:
\begin{enumerate}
\item $K$ is an algebraic Lie subgroup of the automorphism group of the parabolic geometry $(\ba\to M,\om)$ acting transitively on $M$ and we denote by $H$ the stabilizer of a point $x\in M$.
\item There is $u\in \ba$ covering $x$ and reductive Levi decomposition $K=\exp(\fn)\rtimes \bar G$ such that if we define the subgroups
$$\exp(\fn_0):=\{\exp(X)\in \exp(\fn): \exp(X)(u)\in uG_0\},$$
$$\bar G_0:=\{\bar g\in \bar G: \bar g(u)\in uG_0\},$$ 
and 
$$\exp(\bar \fp_+):=\{\bar g\in \bar G: \bar g(u)\in u\exp(\fp_+)\},$$
then $H$ is semidirect product of $\exp(\fn_0)$ and the parabolic subgroup $\bar P$ of $\bar G$ with reductive Levi decomposition $\bar P:=\bar G_0\ltimes \exp(\bar \fp_+)$.
\item There is $\bar s\in \bar G_0$ such that $\bar s(u)=us$ for $s\in G_0$ acting as $-\id$ on $\fg_{-1}.$
\item There is submanifold $\bar M$ of $\bar G/\bar P$ such the flat model $(\bar G\to \bar G/\bar P,\omega_{\bar G})$ restricts to non--homogeneous symmetric parabolic geometry of type $(\bar G,\bar P)$ on $\bar M$.
\end{enumerate}
Then the parabolic geometry 
$$(\ba|_{\exp(\fn)/\exp(\fn_0) \times \bar M}\to \exp(\fn)/\exp(\fn_0)\times \bar M,\om|_{\exp(\fn)/\exp(\fn_0)\times \bar M})$$ 
is non--flat non--homogeneous symmetric parabolic geometry of type $(G,P)$.
\end{prop*}
\begin{proof}
It follows from the assumptions (2) and (3) that the flat model $(\bar G\to \bar G/\bar P,\omega_{\bar G})$ is symmetric. Moreover, from \cite[Sections 3 and 4]{GZ-Lie} follows, that $(\ba\to M,\om)$ is symmetric parabolic geometry and the set of symmetries at $x$ contains a subset isomorphic to $s\exp(\bar \fp_+)$. Therefore the condition (4) implies that the parabolic geometry $(\ba|_{\exp(\fn)/\exp(\fn_0)\times \bar M}\to \exp(\fn)/\exp(\fn_0)\times \bar M,\om|_{\exp(\fn)/\exp(\fn_0)\times \bar M})$ is symmetric, because the set of symmetries at the points $(\exp(X)\exp(\fn_0),\bar x)\in \exp(\fn)/\exp(\fn_0)\times \bar M$ clearly contains the set of symmetries of  $(\bar G|_{\bar M} \to \bar M, \om_{\bar G}|_{\bar M})$ at the points $\bar x$.
\end{proof}

Let us now discuss, when the conditions (1)--(4) of the Proposition \ref{prop} can be satisfied. 

Firstly, the condition (1) posses only topological restrictions on $M$, $G$ and $P$ that are not restrictive. There is a construction in \cite[Section 3]{GZ-DGA} that transforms non--flat $K$--homogeneous parabolic geometry $(\ba\to M,\om)$ into a parabolic geometry satisfying in addition the condition (1), after a sufficient algebraic completion of $G$ and $P$ and covering of $M$.

On the other hand, the conditions (2) and (3) are highly restrictive for non--flat geometries. We know from \cite{GZ-Lie} that not all types of parabolic geometries can admit a symmetry at point with non--trivial curvature and there is even less types of parabolic geometries that admit more then one symmetry at one point, i.e., non--trivial $\exp(\bar \fp_+)$, see the tables in \cite{GZ-Lie}. Moreover, the non--flat homogeneous symmetric parabolic geometries do not satisfy the condition (2), in general. However, in the article \cite[Section 6 (second construction)]{GZ-Lie}, there is a construction of parabolic geometry $(\ba\to M,\om)$ of type $(G,P)$ satisfying the conditions (1),(2),(3) under the following conditions.

\begin{lemma*}\label{cor}
Suppose the type $(G,P)$ of parabolic geometries satisfies the following conditions:
\begin{itemize}
\item There is $s\in G_0$ acting as $-\id$ on $\fg_{-1}$ and acting as $\id$ on some component of the harmonic curvature of parabolic geometries of type $(G,P)$.
\item The lowest weight $\mu$ in the component of the harmonic curvature, on which $s\in G_0$ acts as $\id$, is preserved by the Cartan involution of complexification of $\fg$.
\end{itemize}
Then there is non--flat $K$--homogeneous parabolic geometry $(\ba\to M,\om)$ of type $(G,P)$ satisfying the conditions (1),(2),(3) of Proposition \ref{prop} for $K$ being the automorphism group of $(\ba\to M,\om)$ and $\mu$ being its curvature. 
\end{lemma*}

Motivated by the construction of non--homogeneous flat examples, we study, whether these geometries satisfy the condition (4) of the Proposition \ref{prop}, when we remove two points from the flat model $(\bar G\to \bar G/\bar P,\bar \om)$. We know that removing two points in the case $dim(\bar G/\bar P)=1$ leads to homogeneous parabolic geometry. Therefore we need to consider the cases, when $dim(\bar G/\bar P)>1$.
If we look in the tables in the article \cite{GZ-Lie}, we get that there are only two series of possible types $(G,P)$ (up to covering) satisfying the conditions of the Lemma \ref{cor} admitting $dim(\bar G/\bar P)$ to be greater than one. Let us point out that we need to choose the projectivizations of the groups in order to satisfy the condition (3) for $n$ odd.

\begin{enumerate}
\item[(A)] Consider $G=PGl(n+1,\mathbb{R})$ and $P$ the stabilizer of the flag $e_1\subset e_1\wedge e_2 \subset e_1\wedge \dots \wedge e_l$ in $\mathbb{R}^{n+1}$ for $n\geq 2l-1$, $l>3$, where $e_1,\dots, e_{n+1}$ is the standard basis of $\R^{n+1}$. Then the group $K$ of the non--flat $K$--homogeneous parabolic geometry $(\ba\to M,\om)$ from Lemma \ref{cor} is (as a set) represented by the matrices from $PGl(n+1,\mathbb{R})$ of the form
$$\begin{pmatrix}
L'_{1,1} & L'_{1,2}  & 0 & 0 & \dots & 0  & 0 \cr 
L'_{2,1}  & L'_{2,2}  & 0 & 0& \dots &0 &   0 \cr
N_{3,1} & N_{3,2} & R_3 &0 & \dots  & 0 & 0 \cr
N_{4,1} &N_{4,2} & Z_{4,3} & L_{1,1} & \dots &L_{1,n-3} &0 \cr
 \vdots & \vdots & \vdots & \vdots & \ddots & \vdots & \vdots \cr
N_{n,1} & N_{n,2} & Z_{n,3} & L_{n-3,1}& \dots & L_{n-3,n-3} & 0 \cr 
N_{n+1,1} & N_{n+1,2} & N_{n+1,3} & N_{n+1,4}& \dots &  N_{n+1,n} & R_{n+1}\end{pmatrix},$$
where all the entries are real numbers such that the equalities 
$$(det(L')R_3det(L)R_{n+1})^2=1,\ \ \ det(L')R_3^{-3}R_{n+1}=1$$
hold for the submatrices $L$, $L'$ formed from elements $L_{i,j}$, $L_{i,j}'$.  This means that $\bar G\cong L'\times L$  is the reductive Levi subgroup of $K$, and the unipotent radical corresponds to $N$ and $Z$. The result of multiplication of two elements $\exp(X_1)\bar g_1$, $\exp(X_2)\bar g_2\in \exp(\fn)\ltimes \bar G$ is $\exp(C(X_1,\Ad_{\bar g_1}(X_2)))\bar g_1\bar g_2\in \exp(\fn)\ltimes \bar G$, where $C(-,-)$ represents the Baker--Campbell--Hausdorff--formula for the nilpotent Lie algebra $\fn$. The difference between the Lie bracket in $\fn$ and the Lie bracket in $\frak{sl}(n+1,\mathbb{R})$ of the matrices representing the elements of $\fn$ is precisely the lowest weight of the harmonic curvature of the parabolic geometries of type $(G,P)$, which takes entries in $N_{3,1}$ and $N_{3,2}$ slots and has values in $N_{n+1,3}$ slot.

The subgroup $\exp(\fn_0)$ corresponds to $Z$ entries and the parabolic subgroup $\bar P$ is product of the stabilizer $Q'$ of $e_1$ in $L'$ and the stabilizer $Q$ of $e_4\wedge \dots \wedge e_l$ in $L$. Thus $\bar G/\bar P$ is product of $L'/Q'\cong \mathbb{R}P^1$ and the space $L/Q$ of Grassmannians of $(l-3)$--planes in $\mathbb{R}^{n-3}$. Finally, the element $\bar s$ is the diagonal matrix with $(1,-1,1, \dots, 1, -1\dots, -1,1)$ for exactly $l$ appearances of $1$ on the diagonal.

\item[(C)] Consider $G=PSp(2n,\mathbb{R})$ and $P$ the stabilizer of the flag of isotropic subspaces $e_1\subset e_1\wedge e_2 \subset e_1\wedge \dots \wedge e_n$ in $\mathbb{R}^{2n}$ for $n>4$, where $e_1,\dots, e_{n}$ and $f_1,\dots,f_n$ are bases of two maximally isotropic subspaces in $\R^{2n}$ satisfying $\Omega(e_i,f_j)=\delta_j^i$ for the natural symplectic form $\Omega$ preserved by $PSp(2n,\mathbb{R})$. Then the group $K$ of the non--flat $K$--homogeneous parabolic geometry $(\ba\to M,\om)$ from Lemma \ref{cor} is (as a set) represented by the matrices in $PSp(2n,\mathbb{R})$ with block structure $\begin{pmatrix} A & B \cr C & * \end{pmatrix}$ w.r.t. to the bases $e_1,\dots, e_{n}$ and $f_1,\dots,f_n$, where

$$A:=\begin{pmatrix}
L'_{1,1} & L'_{1,2}  & 0 & 0 & \dots & 0  \cr 
L'_{2,1}  & L'_{2,2}  & 0 & 0& \dots &0 \cr
N_{3,1} & N_{3,2} & R_3 &0 & \dots  & 0 \cr
N_{4,1} &N_{4,2} & Z_{4,3} & L_{1,1} & \dots &L_{1,n-3} \cr
 \vdots & \vdots & \vdots & \vdots & \ddots & \vdots \cr
N_{n,1} & N_{n,2} & Z_{n,3} & L_{n-3,1}& \dots & L_{n-3,n-3} \cr 
\end{pmatrix},$$
$$B:=\begin{pmatrix}
0 & 0 & 0 & 0& \dots & 0\cr 
0 & 0 & 0 & 0& \dots &  0 \cr 
0 &0 & 0 & 0& \dots &  0 \cr 
0 & 0 & 0 & L_{1,n-2}& \dots &  L_{1,2n-6} \cr 
 \vdots & \vdots & \vdots & \vdots & \ddots & \vdots \cr
0 & 0 & 0 & L_{n-3,n-2}& \dots &  L_{n-3,2n-6} \cr \end{pmatrix},$$
$$C:=\begin{pmatrix}
N_{n+1,1} & N_{n+2,1} & N_{n+3,1} & N_{n+4,1}& \dots &  N_{2n,1} \cr 
* & N_{n+2,2} & N_{n+3,2} & N_{n+4,2}& \dots &  N_{2n,2} \cr 
* &* & N_{n+3,3} & N_{n+4,3}& \dots &  N_{2n,3}\cr 
* & * & * & L_{n-2,1}& \dots &  L_{n-2,n-3} \cr 
 \vdots & \vdots & \vdots & \vdots & \ddots & \vdots  \cr
* & * & * & L_{2n-6,1}& \dots &  L_{2n-6,2n-6} \cr \end{pmatrix},$$
where $*$ entries are uniquely determined by the structure of $Sp(2n,\mathbb{R})$, the matrix $L$ formed by elements $L_{i,j}$ is contained in $CSp(2n-6,\mathbb{R})$ and all remaining entries are real numbers such that the equality
$$det(L')R_3^{-4}=1$$
holds for the submatrix $L'$ formed from elements $L_{i,j}'$.  This means that $\bar G\cong L'\times L$  is the reductive Levi subgroup of $K$, and the unipotent radical corresponds to $N$ and $Z$. The result of multiplication of two elements $\exp(X_1)\bar g_1$, $\exp(X_2)\bar g_2\in \exp(\fn)\ltimes \bar G$ is $\exp(C(X_1,\Ad_{\bar g_1}X_2))\bar g_1\bar g_2\in \exp(\fn)\ltimes \bar G$, where $C(-,-)$ represents the Baker--Campbell--Hausdorff--formula for the nilpotent Lie algebra $\fn$. The difference between the Lie bracket in $\fn$ and the Lie bracket in $\frak{sp}(2n,\mathbb{R})$ of the matrices representing the elements of $\fn$ is precisely the lowest weight of the harmonic curvature of the parabolic geometries of type $(G,P)$, which takes entries in $N_{3,1}$ and $N_{3,2}$ slots and has values in $N_{n+3,3}$ slot.

The subgroup $\exp(\fn_0)$ corresponds to $Z$ entries and the parabolic subgroup $\bar P$ is product of the stabilizer $Q'$ of $e_1$ in $L'$ and the stabilizer $Q$ of $e_4\wedge \dots \wedge e_n$ in $L$. Thus $\bar G/\bar P$ is product of $L'/Q'\cong \mathbb{R}P^1$ and the space $L/Q$ of the maximally isotropic (w.r.t. $\Omega$) Grassmannians of $(n-3)$--planes in $\mathbb{R}^{2n-6}$. Finally, the element $\bar s$ is the diagonal matrix with $(1,-1,1, \dots, 1)$ on the first $n$ entries of the diagonal.
\end{enumerate}

Since $\bar G/\bar P\cong \mathbb{R}P^1\times L/Q$ is product of two flat models of parabolic geometries in both of the above cases (A) and (C), we remove two points from the flat model $(L \to L/Q, \om_{L})$ and consider $\bar M:= \mathbb{R}P^1\times (L/Q -\{l_1Q,l_2Q\})$ for some $l_1Q,l_2Q \in L/Q$. Then the flat parabolic geometry $(\bar G\to \bar G/\bar P,\om_{\bar G})$ restricts to a parabolic geometry over $\bar M$ of the same type $(\bar G, \bar P)$. Its automorphisms group consists of direct product of $L'$ and those element of $L$ that preserve the set $\{l_1Q,l_2Q\}$. Thus it decomposes into two components according to the fact whether it preserves $l_1Q$ and $l_2Q$ or whether it swaps $l_1Q$ and $l_2Q$. This property restrict also the possible symmetries on $\bar M$ and there is a natural question, whether at least some symmetries on $\bar G/\bar P$ survive the restriction to $\bar M$. There is the following crucial statement.

\begin{thm*} \label{main}
Let $(L,Q)$ be one of the types of parabolic geometries from the above series (A) or (C). Then $\bar M:= \mathbb{R}P^1\times (L/Q -\{l_1Q,l_2Q\})$ satisfies the condition (4) of the Proposition \ref{prop} if and only if there is $q\in Q$ such that $q l_1^{-1}l_2Q=e_4\wedge \dots \wedge e_{l-1}\wedge e_{l+1}$ in the case (A) or $q l_1^{-1}l_2Q=e_4\wedge \dots \wedge e_{n-1}\wedge f_{n}$ in the case (C).
\end{thm*}

The proof of the Theorem \ref{main} is fairly technical and we give the proof in the next section.
In fact, the proof of the Theorem \ref{main} is equivalent to the proof of the following statement.

\begin{cor*} \label{main2}
Let $(L,Q)$ be one of the types of parabolic geometries from the above series (A) or (C). Then the flat model $(L\to L/Q,\om_L)$ restricts to symmetric parabolic geometry on $L/Q -\{l_1Q,l_2Q\}$ if and only if there is $q\in Q$ such that $q l_1^{-1}l_2Q=e_4\wedge \dots \wedge e_{l-1}\wedge e_{l+1}$ in the case (A) or $q l_1^{-1}l_2Q=e_4\wedge \dots \wedge e_{n-1}\wedge f_{n}$ in the case (C).
\end{cor*}

Let us give the geometric interpretation of the condition in the Theorem \ref{main} and interpret the condition for the existence of preserving symmetry from Lemma \ref{l5}.

\begin{cor*}
Let $(L,Q)$ be one of the types of parabolic geometries from the above series (A) or (C). Then $\bar M:=  \mathbb{R}P^1\times (L/Q -\{l_1Q,l_2Q\})$ satisfies the condition (4) of the Proposition \ref{prop} if and only if the subspaces $W_1$ and $W_2$ corresponding to $l_1Q$ and $l_2Q$ have intersection of dimension $dim(W_1)-1=dim(W_2)-1.$ There is symmetry preserving the subspaces $W_1$ and $W_2$ at the point of $L/Q -\{l_1Q,l_2Q\}$ corresponding to subspace $W$ if and only if the intersection $W\cap (W_1+W_2)$ is contained in $W_1$ or $W_2$.
\end{cor*}

The automorphisms group of the parabolic geometry $(\ba|_{\exp(\fn)/\exp(\fn_0) \times \bar M}\to \exp(\fn)/\exp(\fn_0)\times \bar M,\om|_{\exp(\fn)/\exp(\fn_0)\times \bar M})$ in the case (A) for $\bar M:= \mathbb{R}P^1\times (L/Q -\{e_4\wedge \dots \wedge e_{l-1}\wedge e_{l},e_4\wedge \dots \wedge e_{l-1}\wedge e_{l+1}\})$ has two components. The component of identity consists of (semidirect) product of $L'$, $\exp(\fn)$ and the following matrices in $L$:

$$\begin{pmatrix}
 L_{1,1} & \dots &L_{1,l} & L_{1,l+1} & L_{1,l+2}& \dots & L_{1,n-3} \cr
 \vdots & \ddots & \vdots & \vdots&  \vdots & \ddots & \vdots   \cr
 L_{l-1,1}& \dots & L_{l-1,l} & L_{l-1,l+1} & L_{l-1,l+2}& \dots & L_{l-1,n-3}\cr 
 0& \dots & {\pmb L_{l,l}} &0 &L_{l,l+2}& \dots & L_{l,n-3}  \cr 
 0& \dots &0&{\pmb L_{l+1,l+1}}& L_{l+1,l+2}& \dots & L_{l+1,n-3} \cr 
 \vdots & \ddots & \vdots & \vdots& \vdots & \ddots & \vdots \cr
 0& \dots & 0& 0& L_{n-3,l+2} &\dots & L_{n-3,n-3} \cr 
\end{pmatrix}.$$
The other component consists of (semidirect) product of $L'$, $\exp(\fn)$ and the following matrices in $L$:
$$\begin{pmatrix}
 L_{1,1} & \dots &L_{1,l} & L_{1,l+1} & L_{1,l+2}& \dots & L_{1,n-3} \cr
 \vdots & \ddots & \vdots & \vdots&  \vdots & \ddots & \vdots   \cr
 L_{l-1,1}& \dots & L_{l-1,l} & L_{l-1,l+1} & L_{l-1,l+2}& \dots & L_{l-1,n-3}\cr 
 0& \dots &0&{\pmb L_{l+1,l+1}} &L_{l+1,l+2}& \dots & L_{l+1,n-3}  \cr 
 0& \dots &{\pmb L_{l,l}}&0& L_{l,l+2}& \dots & L_{l,n-3} \cr 
 \vdots & \ddots & \vdots & \vdots& \vdots & \ddots & \vdots \cr
 0& \dots & 0& 0& L_{n-3,l+2} &\dots & L_{n-3,n-3} \cr 
\end{pmatrix}.$$

The automorphisms group of the parabolic geometry $(\ba|_{\exp(\fn)/\exp(\fn_0) \times \bar M}\to \exp(\fn)/\exp(\fn_0)\times \bar M,\om|_{\exp(\fn)/\exp(\fn_0)\times \bar M})$ in the case (C) for $\bar M:= \mathbb{R}P^1\times (L/Q -\{e_4\wedge \dots \wedge e_{n-1}\wedge e_{n},e_4\wedge \dots \wedge e_{n-1}\wedge f_{n}\})$ has two components. The component of identity consists of (semidirect) product of $L'$, $\exp(\fn)$ and the following matrices in $L$:
$$\begin{pmatrix}
L_{1,1} & \dots  & L_{1,n} & L_{1,n+1} & \dots  & L_{1,2n-1}  & L_{1,2n} \cr 
\vdots  & \ddots  & \vdots &\vdots  & \ddots  & \vdots &   \vdots \cr
L_{n-1,1} & \dots & L_{n-1,n} &L_{n-1,n+1} & \dots  & L_{n-1,2n-1}  & L_{n-1,2n} \cr
0 & \dots & {\pmb L_{n,n}} & L_{n,n+1} & \dots &L_{n,2n-1} &0 \cr
0 & \dots & 0 & L_{n+1,n+1}& \dots & L_{n-2,n-3} & 0 \cr 
 \vdots & \ddots & \vdots & \vdots & \ddots & \vdots & \vdots \cr
0 & \dots & 0 & L_{2n-1,n+1}& \dots & L_{2n-1,2n-1} & 0 \cr 
0 & \dots & 0 & L_{2n,n+1}& \dots & L_{2n,2n-1} & {\pmb L_{2n,2n}} \cr 
 \end{pmatrix}.$$
The other component consists of product of $L'$, $\exp(\fn)$ and the following matrices in $L$:
$$\begin{pmatrix}
L_{1,1} & \dots  & L_{1,n} & L_{1,n+1} & \dots  & L_{1,2n-1}  & L_{1,2n} \cr 
\vdots  & \ddots  & \vdots &\vdots  & \ddots  & \vdots &   \vdots \cr
L_{n-1,1} & \dots & L_{n-1,n} &L_{n-1,n+1} & \dots  & L_{n-1,2n-1}  & L_{n-1,2n} \cr
0 & \dots & 0 & L_{2n,n+1} & \dots &L_{2n,2n-1} &{\pmb L_{2n,2n}} \cr
0 & \dots & 0 & L_{n+1,n+1}& \dots & L_{n-2,n-3} & 0 \cr 
 \vdots & \ddots & \vdots & \vdots & \ddots & \vdots & \vdots \cr
0 & \dots& 0 & L_{2n-1,n+1}& \dots & L_{2n-1,2n-1} & 0 \cr 
0 & \dots & {\pmb L_{n,n}} & L_{n,n+1}& \dots & L_{n,2n-1} & 0 \cr 
 \end{pmatrix}.$$
Therefore, there is the following characterization of orbits of the automorphisms group in $\exp(\fn)/\exp(\fn_0)\times \bar M$.

\begin{prop*}
In the case (A) or (C), the points $(\exp(X_1)\exp(\fn_0),l_1'Q',W_3)$ and $(\exp(X_2)\exp(\fn_0),l_2'Q',W_4)$ for the Grasmainans $W_3,W_4$ in $L/Q-\{W_1,W_2\}$ are points in the same orbit of the automorphism group of the parabolic geometry $(\ba|_{\exp(\fn)/\exp(\fn_0) \times \bar M}\to \exp(\fn)/\exp(\fn_0)\times \bar M,\om|_{\exp(\fn)/\exp(\fn_0)\times \bar M})$  if and only if 
\begin{align*}
dim(W_3\cap W_2\cap W_1)&=dim(W_4\cap W_2\cap W_1),\\
dim(W_3\cap (W_2+ W_1))&=dim(W_4\cap (W_2+W_1)),\\
dim(W_3\cap (W_2\sqcup W_1))&=dim(W_4\cap (W_2\sqcup W_1)),
\end{align*}
where $W_2 \sqcup W_1$ is the union of $W_2$ and $W_1$ as algebraic sets.
\end{prop*}

\section{The proof of Theorem \ref{main}}

In order to proof the Theorem \ref{main}, it suffices to prove the Corollary \ref{main2}. So we will assume that $(L,Q)$ is one of the types of parabolic geometries from the above series (A) or (C).
Since $L$ acts by automorphisms of the flat model $(L\to L/Q,\om_L)$ from the left, the restriction of $(L\to L/Q,\om_L)$ to $L/Q-\{l_1Q, l_2Q\}$ is isomorphic to restriction of $(L\to L/Q,\om_L)$ to $L/Q-\{ll_1Q, ll_2Q\}$ for all $l\in L$. 

Therefore we can choose $l=ql_1^{-1}$ for some $q\in Q$ and work with the parabolic geometry on $L/Q-\{eQ,q l_1^{-1}l_2Q\}$. Therefore the non--isomorphic restrictions of $(L\to L/Q,\om_L)$ to $L/Q-\{l_1Q, l_2Q\}$ are parametrized by the double coset space $Q\backslash L/Q$.
We will find a suitable representative $v\in L$ of the classes in $Q\backslash L/Q$ and investigate the symmetries on the restrictions of $(L\to L/Q,\om_L)$ to $L/Q -\{eQ, vQ\}$.

The elements of the Lie algebra $\fl$ of $L$, which are diagonal in the bases $e_1,\dots,e_{n+1}$ or $e_1,\dots,f_{n}$, respectively, form the Cartan subalgebra of the Lie algebra $\fl$. The Lie algebra $\fq$ of $Q$ is a standard parabolic subalgebra of $\fl$ for this Cartan subalgebra and we denote by $\fl_{-1}\oplus \fl_0\oplus \fl_1$ the corresponding $|1|$--grading of $\fl$. Then the subgroups $W(\fl)$, $W(\fl_0)$ generated by elements of $L$, $L_0$, which permute the elements of the bases $e_1,\dots,e_{n+1}$ or $e_1,\dots,f_{n}$, respectively, induce the Weyl groups of $\fl$, $\fl_0$. Let us recall that there are representatives of the classes of $W(\fl)/W(\fl_0)$ encoded by the Hasse diagram $\mathcal{W}^{\fq}$ of the parabolic subalgebra $\fq$, which define the decomposition of $L/Q$ into Schubert cells, see \cite[Section 3.2.19]{parabook}.

\begin{lemma*}
Each element of $L/Q$ can be uniquely written as $\exp(Z)wQ$ for $w\in \mathcal{W}^{\fq}$ and $Z\in \Ad_w^{-1}(\fl_{-1})\cap \fb_+$, where $\fb_+\subset \fq$ is the sum of all positive root spaces in $\fl$. The dimension of $\Ad_w^{-1}(\fl_{-1})\cap \fb_+$ is equal to the length of $w$ in $\mathcal{W}^{\fq}$.
\end{lemma*}

Consequently, the double coset space $Q\backslash L/Q$ is finite and is in bijective correspondence with the double coset space $W(\fl_0)\backslash W(\fl)/W(\fl_0)$. Therefore we can represent the classes of $Q\backslash L/Q$ by the shortest elements in the Haase diagram $\mathcal{W}^\fq$ from the class in $W(\fl_0)\backslash W(\fl)/W(\fl_0)$.

We start the investigation of the symmetries of the restriction of $(L\to L/Q,\om_L)$ to $L/Q -\{eQ, vQ\}$ on the smallest cell $\exp(Z)wQ$ for $w\in \mathcal{W}^{\fq}$ of the length $1$. In the case $(A)$, there is unique $w$ of length $1$, which is the simple reflection over the $(l-3)^{rd}$ simple root that corresponds to swapping $e_{l}$ and $e_{l+1}$, and $Z$ is contained in the root space of the $(l-3)^{rd}$ simple root of $\fl$. In the case $(C)$, there is unique $w$ of length $1$, which is the simple reflection over the $(n-3)^{rd}$ simple root that corresponds to swapping $e_{n}$ and $f_{n}$, and $Z$ is contained in the root space of the $(n-3)^{rd}$ simple root of $\fl$.

\begin{lemma*} \label{l1}
Let $(L,Q)$ be the type of parabolic geometry from (A) or (C), let $w$ be the unique element of $\mathcal{W}^{\fq}$ 
of the length $1$ and let $v$ be the shortest element in $\mathcal{W}^{\fq}$ representing a class in $Q\backslash L/Q$. If the length of $v$ is greater than $1$, then there is no symmetry at the points $\exp(Z)wQ, Z\neq 0$ of $L/Q -\{eQ, vQ\}$.
\end{lemma*}
\begin{proof}
There is a symmetry at the point $\exp(Z)wQ$ preserving the points $eQ$ and $vQ$ if and only if there is $Y\in \fl_1$ such that $$\exp(Z)ws\exp(Y)(\exp(Z)w)^{-1}\in Q$$ and simultaneously $$v^{-1}\exp(Z)ws\exp(Y)(\exp(Z)w)^{-1}v\in Q$$ hold. Since $N_{L}(Q)=Q$,  $v^{-1}wsw^{-1}v\in Q$ and $\exp(\Ad_{wsw^{-1}}^{-1}(Z))=\exp(-Z)$ hold for both types (A) and (C), these two conditions are equivalent to the conditions $$\Ad_{\exp(Z)w}(Y)\in \fq$$ and simultaneously $$\exp(\Ad_v^{-1}\Ad_{\exp(-Z)w}(Y))\exp(-2\Ad_v^{-1}(Z))\in Q.$$ From the structure of $\mathcal{W}^{\fq}$ follows that $\Ad_v^{-1}(Z)$ is a non--zero element of $\fl_{-1}$, while the condition $\Ad_{\exp(Z)w}(Y)\in \fq$ implies that $\Ad_{\exp(-Z)w}(Y)$ has trivial component in the root space of $(l-3)^{rd}$ or $(n-3)^{rd}$ simple root, respectively. Therefore, there is symmetry at $\exp(Z)wQ$ preserving the points $eQ$ and $v Q$ only if $Z= 0$.

There is a symmetry at $\exp(Z)wQ$ swapping the points $eQ$ and $v Q$ if and only if there is $Y\in \fl_1$ such that the condition $$\exp(Z)ws\exp(Y)(\exp(Z)w)^{-1}v\in Q$$
holds. This condition is equivalent to the condition $$\exp(\Ad_{\exp(Z)w}(Y))v=\exp(\Ad_w(Y)+[Z,\Ad_w(Y)]+1/2[Z,[Z,\Ad_w(Y)]])v\in Q.$$
Since the right multiplication by elements of $W(\fl)$ acts by swapping columns in the matrix $\exp(\Ad_w(Y)+[Z,\Ad_w(Y)]+1/2[Z,[Z,\Ad_w(Y)]])$, 
the decisive for the existence of the symmetry are the entries on the diagonal of $\exp(\Ad_w(Y)+[Z,\Ad_w(Y)]+1/2[Z,[Z,\Ad_w(Y)]])$. But there are numbers $1$ on the diagonal except the $(l-2)^{nd}$ and $(l-3)^{rd}$ position in the case (A) and $(n-3)^{nd}$ and $2(n-3)^{rd}$ position in the case (C) that both depend on $[Z,\Ad_w(Y)]$. Therefore, if the length of $v$ is greater than one, there is no swapping symmetry at $\exp(Z)wQ$. 
\end{proof}

Therefore it remains to investigate the symmetries of the restriction of the flat model to $L/Q -\{eQ, vQ\}$ for the unique element $v$ of $\mathcal{W}^{\fq}$ of the length $1$. In this case, we can again use the decomposition of $L/Q$ into Schubert cells to show, when there is a symmetry preserving the points $eQ$ and $vQ$.

\begin{lemma*}\label{l5}
Let $(L,Q)$ be the type of parabolic geometry from (A) or (C) and let $v$ be the unique element of $\mathcal{W}^{\fq}$ of the length $1$. Then there is a symmetry at point $\exp(Z)wQ$ of $L/Q -\{eQ, vQ\}$ preserving the points $eQ$ and $vQ$ if and only if $Z$ has trivial component in the root space of the $(l-3)^{rd}$ or $(n-3)^{rd}$ simple root of $\fl$, respectively.
\end{lemma*}
\begin{proof}
Since the conditions $v^{-1}wsw^{-1}v\in Q$ and $\exp(\Ad_{wsw^{-1}}^{-1}(Z))=\exp(-Z)$   from proof of the Lemma \ref{l1} are satisfied for generic $v$ and $w$ and $Z\in \Ad_w^{-1}(\fl_{-1})\cap \fb_+$, the symmetry $\exp(Z)ws(\exp(Z)w)^{-1}$ at the point $\exp(Z)wQ$ of $L/Q -\{eQ, vQ\}$ preserves the points $eQ$ and $vQ$ if and only if $Z$ has trivial component in the root space of the $(l-3)^{rd}$ or $(n-3)^{rd}$ simple root of $\fl$.

It remains to show that there are no preserving symmetries at the other points of $L/Q -\{eQ, vQ\}$. Therefore it suffices to show that if the symmetry 
$$\exp(Z)ws\exp(Y)(\exp(Z)w)^{-1}$$ 
at the point $\exp(Z)wQ$ preserves the points $eQ$ and $vQ$, then $Z$ has a trivial component in the root space of the $(l-3)^{rd}$ or $(n-3)^{rd}$ simple root of $\fl$, respectively.  Let us assume that the conditions $$\Ad_{\exp(Z)w}(Y)\in \fq$$ and simultaneously $$\exp(\Ad_v^{-1}\Ad_{\exp(-Z)w}(Y))\exp(-2\Ad_v^{-1}(Z))\in Q$$ hold. If $Z$ has a non--trivial component in the root space of the $(l-3)^{rd}$ or $(n-3)^{rd}$ simple root of $\fl$, then $\Ad_{\exp(-Z)w}(Y)$ has non--trivial component in the root space of the $(l-3)^{rd}$ or $(n-3)^{rd}$ simple root of $\fl$, too. But $$\Ad_{\exp(-Z)w}(Y)=\Ad_w(Y)-[Z,\Ad_w(Y)]+1/2[Z,[Z,\Ad_w(Y)]]\in \fq$$ follows from the condition $\Ad_{\exp(Z)w}(Y)\in \fq$, and thus $\Ad_w(Y)\in \fq$ has non--trivial component in the root space of the $(l-3)^{rd}$ or $(n-3)^{rd}$ simple root of $\fl$. This is contradiction with the condition $Z\in \Ad_w^{-1}(\fl_{-1})\cap \fb_+$ for $Z$ with a non--trivial component in the root space of the $(l-3)^{rd}$ or $(n-3)^{rd}$ simple root of $\fl$, because $w=v\circ w'$ for some $w'\in \mathcal{W}^\fq$ holds and if the $(l-3)^{rd}$ or $(n-3)^{rd}$ simple root of $\fl$ is in the image of $\Ad_w(\fl_1)$, then the dimension of $\Ad_{w'}^{-1}(\fl_{-1})\cap \fb_+$ is $dim(\Ad_w^{-1}(\fl_{-1})\cap \fb_+)+1$, which is contradiction.
\end{proof}

Therefore it remains to show that there is a symmetry swapping the points $eQ$ and $vQ$ at the points $\exp(Z)wQ\in L/Q$ such that $Z$ has a non--trivial component in the root space of the $(l-3)^{rd}$ or $(n-3)^{rd}$ simple root of $\fl$. We show this as a part of the following lemma that summarizes the previous statements.

\begin{lemma*}
Let $(L,Q)$ be the type of parabolic geometry from (A) or (C) and let $v$ be the unique element of $\mathcal{W}^{\fq}$ of the length $1$. Then there is a symmetry either preserving or swapping $eQ$ and $vQ$ at each $\exp(Z)wQ\in L/Q -\{eQ, vQ\}$. 
\end{lemma*}
\begin{proof}
Suppose $\exp(Z)wQ\in L/Q -\{eQ, vQ\}$ is such that $Z$ has non--trivial component in the root space of the $(l-3)^{rd}$ or $(n-3)^{rd}$ simple root of $\fl$. Since the conditions $v^{-1}wsw^{-1}v\in Q$ and $\exp(\Ad_{wsw^{-1}}^{-1}(Z))=\exp(-Z)$   from proof of the Lemma \ref{l1} are satisfied for generic $v$ and $w$ and $Z\in \Ad_w^{-1}(\fl_{-1})\cap \fb_+$, the symmetry $\exp(Z)ws\exp(Y)(\exp(Z)w)^{-1}$ at the point $\exp(Z)wQ$ of $L/Q -\{eQ, vQ\}$ swaps the points $eQ$ and $vQ$ if and only if $$\exp(\Ad_w(Y)+[Z,\Ad_w(Y)]+1/2[Z,[Z,\Ad_w(Y)]])v\in Q.$$ 

However, $v$ swaps $(l-2)^{nd}$ and $(l-3)^{rd}$ column in the case (A) and $(n-3)^{rd}$ and $2(n-3)^{rd}$ column in the case (C). Therefore there is a symmetry at $\exp(Z)wQ$ if there is $0$ on the $(l-3)^{rd}$ or $2(n-3)^{rd}$ position on diagonal in the matrix $\exp(\Ad_w(Y)+[Z,\Ad_w(Y)]+1/2[Z,[Z,\Ad_w(Y)]])$ and the component of $\Ad_w(Y)$ in $\fl_{-1}$ is contained in the root space of the minus $(l-3)^{rd}$ or $(n-3)^{rd}$ simple root of $\fl$. If  $Z\in \Ad_w^{-1}(\fl_{-1})\cap \fb_+$ has non--trivial component in the root space of the $(l-3)^{rd}$ or $(n-3)^{rd}$ simple root of $\fl$ and there is the duality between the positive and negative roots, there is $Y\in \fl_1$ such that $\Ad_w(Y)$ is contained in the root space of the minus $(l-3)^{rd}$ or $(n-3)^{rd}$ simple root of $\fl$. If $Y$ is anti--proportional to the component of $Z$  in the root space of the $(l-3)^{rd}$ or $(n-3)^{rd}$ simple root of $\fl$, then there is $0$ on the $(l-3)^{rd}$ or $2(n-3)^{rd}$ position on diagonal in the matrix $\exp(\Ad_w(Y)+[Z,\Ad_w(Y)]+1/2[Z,[Z,\Ad_w(Y)]])$ and the symmetry $\exp(Z)ws\exp(Y)(\exp(Z)w)^{-1}$ at the point $\exp(Z)wQ$ of $L/Q -\{eQ, vQ\}$ swaps the points $eQ$ and $vQ$.
\end{proof}

\end{document}